\documentclass{amsart}

\usepackage{amsfonts, amssymb, amsmath, eucal, verbatim, amsthm, amscd, enumerate}
\usepackage{graphicx}

\newtheorem{theorem}{Theorem}[section]
\newtheorem{lemma}[theorem]{Lemma}

\newtheorem{corollary}[theorem]{Corollary}
\theoremstyle{definition}

\theoremstyle{remark}
\newtheorem*{remark}{Remark}

\numberwithin{equation}{section}
\def\Re{\mathrm{Re}}

\def\R{{\mathbb R}}

\def\C{{\mathbb C}}

\begin{document}

\date{7 December 2010}

\author{Neal Bez}
\address{Neal Bez, School of Mathematics, The Watson Building, University of Birmingham, Edgbaston,
Birmingham, B15 2TT, England} \email{n.bez@bham.ac.uk}
\subjclass{Primary  35B45; Secondary  35L05}

\keywords{Strichartz estimates; wave equation; sharp constants}

\author{Keith M. Rogers}
\address{Keith Rogers, Instituto de Ciencias Matematicas CSIC-UAM-UC3M-UCM,
Madrid 28049, Spain} \email{keith.rogers@icmat.es}
\begin{thanks} {The second author is supported in part by the Spanish grant MTM2010-16518.}
\end{thanks}
\title[A sharp Strichartz estimate for the wave equation]{A sharp Strichartz estimate for the wave equation with data in the energy space}
\maketitle

\begin{abstract}
We prove a sharp bilinear estimate for the wave equation from which
we obtain the sharp constant in the Strichartz estimate which
controls the $L^{4}_{t,x}(\R^{5+1})$ norm of the solution in terms
of the energy. We also characterise the maximisers.
\end{abstract}

\section{Introduction}

For $d\ge 2$, we consider the wave equation $
\partial_{tt}u = \Delta u$ on $\R^{d+1}$.
Strichartz \cite{st} proved that
\begin{equation}\label{half}
\|u\|_{L^p_{t,x}(\R^{d+1})} \leq C \Big(\|u(0)\|^2_{\dot{H}^{\frac{1}{2}}(\mathbb{R}^d)}+\|\partial_t
u(0)\|^2_{\dot{H}^{-\frac{1}{2}}(\mathbb{R}^d)}\Big)^{1/2},\quad p = \frac{2(d+1)}{d-1}.
\end{equation}
Here, $\dot{H}^s(\mathbb{R}^d)$ denotes the homogeneous Sobolev
space with norm
$$
\| f \|_{\dot{H}^s(\mathbb{R}^d)} = \|(-\Delta)^{s/2} f\|_{L^2(\mathbb{R}^d)},
$$
where $(-\widehat{\Delta)^{s/2} f(}\xi) =|\xi|^{s} \widehat{f}(\xi),
$ and $\,\widehat{\,}\,\,$ is the Fourier transform defined by
$$
\widehat{f}(\xi)=\int_{\R^d} f(x)\exp(-ix\cdot\xi)\,dx.
$$
Foschi \cite{Foschi} found the sharp constant in \eqref{half} for
$d=3$ and a characterisation of the data $(u(0),\partial_tu(0))$ for
which the constant is attained.

For $d\ge 3$, by interpolation and Sobolev embedding,  \eqref{half}
yields
\begin{equation}\label{sobst} \|u\|_{L^p_{t,x}(\R^{d+1})}
\leq C\Big(\|\nabla
u(0)\|^2_{L^2(\mathbb{R}^d)}+\|\partial_tu(0)\|^2_{L^2(\mathbb{R}^d)}\Big)^{1/2}, \quad p =
\frac{2(d+1)}{d-2}.
\end{equation}
This estimate has found a great deal of application in the nonlinear
theory. Indeed, the standard blow-up criterion for the focussing
energy-critical equation is written in terms of the
$L^p_{t,x}(\mathbb{R}^{d+1})$ norm with $p = \frac{2(d+1)}{d-2}$
(see for example \cite{keme}). Thus, it seems of interest to know
the data which maximise~\eqref{sobst}. That such data exist is due
to Bulut \cite{bulut} (see also \cite{bage}).

In this article we prove the following sharp bilinear inequality for
the one-sided wave propagator $e^{it\sqrt{-\Delta}}$ given by
$$
e^{it\sqrt{-\Delta}}f(x) = \frac{1}{(2\pi)^d}\int_{\R^d}
\widehat{f}(\xi) \exp\big(i(x\cdot\xi+t|\xi|)\big)\,d\xi.
$$
The solution to the wave equation can be written as $u = u_+ + u_-$,
where $$ u_+(t) = e^{it\sqrt{-\Delta}}f_+\quad \text{and}\quad
u_-(t) = e^{-it\sqrt{-\Delta}}f_-,$$ and\footnote{in \cite{Foschi}
the functions $f_{\pm}$ are defined slightly differently}
$$
u(0)=f_++f_-\quad \text{and}\quad
\partial_tu(0)=i\sqrt{-\Delta}\,\big(f_+-f_-\big).
$$
From this bilinear inequality, we will deduce the sharp constant in
the energy-Strichartz estimate \eqref{sobst} for $d=5$, and
characterise the maximising data.

\begin{theorem}\label{fq2} Let $d\ge 2$. Then the inequality
\begin{align}\label{e:FoschiCarneirohighd}
\big\| e^{it\sqrt{-\Delta}}f_1
\,e^{it\sqrt{-\Delta}}f_2\big\|_{L^{2}_{t,x}(\R^{d+1})}^{2}\nonumber&\\\begin{split}
&\!\!\!\!\!\!\!\!\!\!\!\!\!\!\!\!\!\!\!\!\!\!\!\!\!\!\!\!\!\!\!\!\!\!\!\!\!\!\!\!\!\!\!\!\!\!\!\!\!\!
\leq {\emph{W}}(d,2) \int_{\R^{2d}}
|\widehat{f}_1(\xi_1)|^2|\widehat{f}_2(\xi_2)|^2
|\xi_1|^{\frac{d-1}{2}}|\xi_2|^{\frac{d-1}{2}}\bigg(1-\frac{\xi_1\cdot\xi_2}{|\xi_1||\xi_2|}\bigg)^{\frac{d-3}{2}}\,d\xi_1d\xi_2
\end{split}\end{align}
holds with constant given by
$$
{\emph{W}}(d,2) = 2^{-\frac{d-1}{2}}(2\pi)^{-3d+1}|\mathbb{S}^{d-1}|.
$$
For $d\ge 3$, the constant is sharp and is attained  if and only if
\begin{equation*}|\xi|\widehat{f_j}(\xi)=\exp(a|\xi|+b\cdot\xi+c_j),\end{equation*}
where $a,c_1,c_2 \in \C$, $b\in\C^d$, $\Re(a)<0$ and
$|\Re(b)|<-\Re(a)$.
\end{theorem}

In particular,  the constant is attained when
$\sqrt{-\Delta}f_1=\sqrt{-\Delta}f_2=(1+|\cdot|^2)^{-\frac{d+1}{2}}$.
In Sections ~\ref{section:Carneirowave} and \ref{characterization}
we prove a $k$-linear generalisation of Theorem \ref{fq2} with sharp
constant ${\mbox{W}}(d,k)$ for $(d,k) \neq (2,2)$.

Estimates which are similar in spirit, but with different \lq null'
weights, were proven by Klainerman and Machedon~\cite{klma, km2,
km3}, among others. They also conjectured that estimates, for
functions with separated angular Fourier supports and with the $L^2$
norm on the left-hand side replaced by an $L^p$ norm, should hold.
For the optimal range of $p$ (modulo the endpoint) this problem was
resolved in the remarkable article of Wolff \cite{wo} (see Tao
\cite{ta0} for the endpoint), following the pioneering work of
Bourgain~\cite{bo}. The $L^2$-version of the null-form conjecture
was resolved in \cite{fokl} and the $L^p$-version
 (modulo the endpoint) in~\cite{ta0, leva, lerova}.

When $d=2$, the power of the angular weight is negative, and this
estimate was implicit in the work of Barcel\'o \cite{ba}. One can
calculate that the integral on the right-hand side of
\eqref{e:FoschiCarneirohighd} is unbounded for integrable
$f_1=\lambda f_2\neq0$.

The power of the angular weight is zero when $d=3$, and in this case
the sharp inequality and characterisation of maximisers is due to
Foschi \cite{Foschi}. The sharp constant in the Strichartz estimate
\eqref{half} and the characterisation of maximisers follows from
this (see \cite{Foschi}).

In contrast with the two dimensional case, when $d\ge 4$ the
estimate \eqref{e:FoschiCarneirohighd} improves if the interacting
waves have overlapping angular Fourier support. In particular, when
$d=5$, we will see that Theorem~\ref{fq2} is stronger than the sharp
energy-Strichartz estimate, which we obtain as a consequence.

\begin{corollary}\label{thecor}
Suppose that $\partial_{tt}u=\Delta u$ on $\R^{5+1}$. Then
\begin{equation*}\label{theest}
\|u\|_{L^{4}(\R^{5+1})} \leq \frac{1}{(8\pi)^{1/2}}\, \Big(\|\nabla
u(0)\|^2_{L^2(\mathbb{R}^5)}+\|\partial_t
u(0)\|_{L^2(\mathbb{R}^5)}^2\Big)^{1/2}.
\end{equation*}
The constant is sharp and is attained  if and only if
\begin{equation}\label{ref}\big(u(0),\partial_tu(0)\big)=\Big(0,(1+|\cdot|^2)^{-\frac{d+1}{2}}\Big),\end{equation}
with $d=5$, modulo the action of the group generated by
\begin{itemize}
\item[(W1)] $u(t,x)\to u(t+t_0,x+x_0)$ with $t_0\in\R$, $x_0\in\R^d$,
\item[(W2)] $u(t,x)\to \lambda_1 u(\lambda_2 t,\lambda_2 x)$ with $\lambda_1,\lambda_2
>0$,
\item[(W3)] $u(t,x)\to e^{i\theta_+} u_+(t,x)+e^{i\theta_-} u_-(t,x)$ with $\theta_+,\theta_-
\in\R$.
\end{itemize}
\end{corollary}

Being a corollary of Theorem~\ref{fq2}, the proof relies heavily on
the Fourier transform, however the Fourier transform makes no
appearance in the final inequality. Indeed the wave equation is
often considered as a real equation, and it would be interesting to
know if \eqref{theest} could be proven without the use of the
complex numbers.

By the conservation of energy, if the initial data is a maximising
pair, then $(u(t),\partial_t u(t))$ must also be a maximising pair.
Thus the evolution of the initial data \eqref{ref} can be described
in terms (W1),(W2) and (W3).

 Other well-known symmetries for the wave equation are
spatial rotations and Lorentzian boosts:
\begin{itemize}
\item[(W4)] $u(t,x)\to u(t,Rx)$ with
$R\in \mbox{SO}(d),$
\item[(W5)] $u(t,x)\to u(\cosh(a)t+\sinh(a)x_1, \cosh(a)x_1+\sinh(a)t,x')$ with
$a\in\R$.
\end{itemize}
In \cite[Theorem 1.7]{Foschi}, it is shown that the maximisers for
\eqref{half} with $d=3$ can be obtained from the action of the group
generated by (W1)--(W5) on the pair $\eqref{ref}$. Thus, the class
of maximisers is larger than that of Corollary \ref{thecor}. This is
explained by the fact that
$$
\Big(\|u(0)\|^2_{\dot{H}^{\frac{1}{2}}(\mathbb{R}^d)}+\|\partial_t
u(0)\|^2_{\dot{H}^{-\frac{1}{2}}(\mathbb{R}^d)}\Big)^{1/2}
$$
is invariant under (W5) whereas the energy is not.


We now dedicate some words to the recent history of the problem and
the structure of the article. In order to do so, we will need to
discuss the closely related Schr\"odinger evolution operator
$e^{it\Delta}$ given by
$$
e^{it\Delta}f(x) =\frac{1}{(2\pi)^d}\int_{\R^d}
\widehat{f}(\xi) \exp\big(i(x\cdot\xi-t|\xi|^2)\big) \,d\xi.
$$
Analogous to \eqref{half}, Strichartz \cite{st} proved that
\begin{equation}\label{stricschro}
\big\| e^{it\Delta} f \big\|_{L^p_{t,x}(\R^{d+1})} \leq C\,
\|f\|_{L^2(\R^d)},\quad p = \frac{2(d+2)}{d},
\end{equation}
which followed work of Stein and Tomas on the Fourier extension
problem on the unit sphere $\mathbb{S}^{d-1}$.

Some decades later, Kunze \cite{ku}  proved the existence of
maximisers for \eqref{stricschro} with $d=1$, and Foschi
\cite{Foschi} found the maximisers when $d=1,2$. This was reproved
via different techniques by Hundertmark--Zharnitsky~\cite{huza} (see
also Bennett \textit{et al} \cite{bebecahu} for an alternative
derivation of the sharp constant when $d=1,2$ using heat-flow
methods).
 Carneiro~\cite{Carneiro} then
developed the ideas of Hundertmark--Zharnitsky in order to prove
analogous results to our forthcoming Theorem \ref{t:FoschiCarneiro}
for the Schr\"odinger operator. That maximisers for
\eqref{stricschro} exist in higher dimensions is due to Shao
\cite{shao} (see also \cite{meve}, \cite{cake}, \cite{beva}).

More recently, Duyckaerts, Merle and Roudenko \cite{dumero} proved
that the
 $L^p_{t,x}(\mathbb{R}^{d+1})$ norm, with $p = \frac{2(d+2)}{d}$, of the solution to the $L^2$-critical nonlinear Schr\"odinger equation is maximised over data with fixed (small) $L^2(\mathbb{R}^d)$  norm.
For $d=1,2$, they used the result of Foschi to calculate the size of
the maximum norm with some precision. In particular, they showed
that the Strichartz norm for the focussing equation with small data
is larger than in the linear case. They remark that parts of their
proof should be flexible enough to treat the energy-critical
Schr\"odinger and wave equations, and Corollary ~\ref{thecor} is a
step in that direction.

Finally, Christ and Shao~\cite{chsh, chsh2} proved the existence of
maximisers for the original Stein--Tomas extension inequality on the
two-dimensional sphere, and that the  maximisers $f$ are necessarily
smooth and satisfy $|f(x)|=|f(-x)|$. For general compact surfaces
and dimensions, Fanelli, Vega and Visciglia \cite{favevi} obtained
the existence of maximisers for the associated extension
inequalities up to the endpoint (at which it is also shown that
existence is not guaranteed in general). In particular, the result
in \cite{favevi} holds for $p>\frac{2(d+2)}{d}$ when extending on
$\mathbb{S}^{d-1}$.

 In Section~\ref{wave1}, we state our
results for the wave propagator in multilinear form and prove
Corollary~\ref{thecor} and some further corollaries for $d=2,3$. In
Section~\ref{section:Carneirowave}, we prove the sharp multilinear
inequality, and  we characterise the maximisers for this inequality
in Section~\ref{characterization}. Finally, in
Section~\ref{section:Carneiro}, we revisit the result of Carneiro
\cite{Carneiro} for the Schr\"odinger evolution operator in order to
make a number of remarks and to provide an alternative proof
following Foschi \cite{Foschi}.

\section{Main results}\label{wave1}

We state our result in full generality (in terms of the
multilinearity). In order to write down an expression for the sharp
constant ${\mbox{W} }(d,k)$ we need the beta function $\mbox{B}$
given by
$$
\mbox{B}(x,y) = \int_0^1 s^{x-1} (1-s)^{y-1}\,ds
$$
for $x,\,y > 0$.

\begin{theorem} \label{t:FoschiCarneiro}
Suppose that $d, k\ge 2$ and let $ \alpha(k) =
\frac{(d-1)(k-1)}{2}-1$. Let $K : (\mathbb{R}^d)^k \rightarrow
[0,\infty)$ be given by
$$
K(\eta) = \bigg( \sum_{1 \leq i<j \leq k} (|\eta_i||\eta_j| - \eta_i \cdot
\eta_j) \bigg)^{1/2}
$$
for $\eta = (\eta_1,\ldots,\eta_k) \in (\mathbb{R}^d)^k$. Then the
inequality
\begin{equation} \label{e:FoschiCarneiro}
\Big\| \prod_{j=1}^k e^{it\sqrt{-\Delta}}f_j
\Big\|_{L^{2}_{t,x}(\mathbb{R}^{d+1})}^{2} \leq {\emph{W} }(d,k) \int_{\R^{kd}}
\prod_{j=1}^k|\widehat{f}_j(\eta_j)|^2|\eta_j| \,K(\eta)^{2\alpha(k)}
\,d\eta
\end{equation}
holds with constant given by
$$
{\emph{W} }(d,k) = 2^{-\frac{d-1}{2}}(2\pi)^{-3d+1}|\mathbb{S}^{d-1}|
$$
if $k=2$, and
$$
{\emph{W} }(d,k) = 2^{-\frac{(d-1)(k-1)}{2}} (2\pi)^{-d(2k-1)+1} |\mathbb{S}^{d-1}|^{k-1} \prod_{j=2}^{k-1} \emph{B}\big(d-1,\alpha(j) + 1\big)
$$
if $k \geq 3$. Whenever $(d,k) \neq (2,2)$ the constant ${\emph{W}
}(d,k)$ is sharp and is attained if and only if
\begin{equation*}|\xi|\widehat{f_j}(\xi)=\exp(a|\xi|+b\cdot\xi+c_j),\end{equation*}
where $a,c_1,\ldots,c_k\in \C$, $b\in\C^d$, $\Re(a)<0$ and
$|\Re(b)|<-\Re(a)$.
\end{theorem}
Theorem \ref{t:FoschiCarneiro} was proven by Foschi in the cases
where $\alpha(k)=0$. This occurs if and only if $(d,k)$ is either
 $(2,3)$ or $(3,2)$ and yields the sharp Strichartz estimates for the one-sided operator when $d=2,3$ and a characterisation of the maximisers.

 The cases where
$\alpha(k)=1$ are also special and this occurs if and only if
$(d,k)$ is
 $(2,5)$, $(3,3)$ or $(5,2)$. We employ a basic yet very useful
observation of Carneiro \cite{Carneiro} to deduce the following
estimates.

\begin{corollary}\label{seven} In two spatial dimensions,
$$
\big\| e^{it\sqrt{-\Delta}} f \big\|_{L^{10}_{t,x}(\R^{2+1})} \leq
\Big(\frac{5}{12\pi^3}\Big)^{1/10} \|f\|_{\dot{H}^{\frac{1}{2}}(\R^2)}^{3/5}
\|f\|_{\dot{H}^1(\R^2)}^{2/5}.
$$
In three spatial dimensions,
$$
\big\| e^{it\sqrt{-\Delta}} f \big\|_{L^{6}_{t,x}(\R^{3+1})} \leq
\Big(\frac{3}{16\pi^3}\Big)^{1/6} \|f\|_{\dot{H}^{\frac{1}{2}}(\R^3)}^{1/3}
\|f\|_{\dot{H}^1(\R^3)}^{2/3}.
$$
In five spatial dimensions,
$$
\big\| e^{it\sqrt{-\Delta}} f \big\|_{L^{4}_{t,x}(\R^{5+1})} \leq
\Big(\frac{1}{24\pi^{2}}\Big)^{1/4}  \|f\|_{\dot{H}^1(\R^5)}.
$$
The constants are sharp and are attained if and only if
$$|\xi|\widehat{f}(\xi) = \exp({a|\xi|+ib\cdot\xi+c}),$$ where
$a,c\in\C$, $b\in\R^d$ and $\Re (a) <0$.
\end{corollary}

\begin{proof}[Proof of Corollary \ref{seven}]
We have that $\alpha(k)=1$. Taking $f_1=\cdots =f_k=f$, the integral
on the right-hand side of \eqref{e:FoschiCarneiro} can be written as
$$
\int_{\R^{kd}}\prod_{\ell=1}^k|\widehat{f}(\eta_\ell)|^2|\eta_\ell|
K(\eta)^{2\alpha(k)}\,d\eta = \mbox{I} - \mbox{II},
$$
where
$$
\mbox{I} = \sum_{1 \leq i<j \leq k} \int_{\R^{kd}}
\prod_{\ell=1}^k|\widehat{f}(\eta_\ell)|^2|\eta_\ell| |\eta_i|
|\eta_j| \,d\eta =
\frac{k(k-1)}{2}(2\pi)^{kd}\,\|f\|_{\dot{H}^{\frac{1}{2}}(\mathbb{R}^d)}^{2(k-2)}
\|f\|_{\dot{H}^1(\mathbb{R}^d)}^4,
$$
and
$$
\mbox{II} = \sum_{1 \leq i<j \leq k} \int_{\R^{kd}}
\prod_{\ell=1}^k|\widehat{f}(\eta_\ell)|^2|\eta_\ell|\, \eta_i \cdot
\eta_j \,d\eta.
$$
As in \cite{Carneiro}, by writing $\xi=(\xi_1,\ldots,\xi_d)$, we
have
\begin{equation}\label{theline}
\int_{\R^{2d}} |\widehat{f}(\eta_i)|^2|\eta_i|
|\widehat{f}(\eta_j)|^2|\eta_j| \, \eta_i \cdot \eta_j
\,d\eta_id\eta_j= \sum_{m=1}^d \Big(\int_{\R^{d}}
|\widehat{f}(\xi)|^2|\xi|\, \xi_m  \,d\xi\Big)^2\ge 0,
\end{equation}
so that
\begin{align*}
\big\| e^{it\sqrt{-\Delta}} f \big\|^{2k}_{L^{2k}_{t,x}(\mathbb{R}^{d+1})} & \leq
{\mbox{W} }(d,k)\,(\mbox{I}-\mbox{II}) \\
& \leq \Big[\frac{k(k-1)}{2}(2\pi)^{kd}{\mbox{W} }(d,k)\Big]\,
\|f\|_{\dot{H}^{\frac{1}{2}}(\mathbb{R}^d)}^{2(k-2)} \|f\|_{\dot{H}^1(\mathbb{R}^d)}^{4}
\end{align*}
with equality at each inequality for the functions with radial
modulus given by
$$\widehat{f}(\xi) = |\xi|^{-1}\exp({a|\xi|+b\cdot\xi+c}),$$ where
$a,c\in\C$, $b\in\C^d$ and $\Re(b)=0$.

It remains to characterise the maximisers. That is to say, to prove
that when $b \in \C^d$ with $\Re (b) \neq 0$, the quantity
$\mbox{II}$ is nonzero. By a rotation we can suppose that
$$\widehat{f}(\xi) = |\xi|^{-1}\exp({a|\xi|+b\cdot\xi+c}),$$ where
$a,c\in\C$, $b\in\C^d$, $\Re (a) <0$, $\Re(b)=(b_1,0,\ldots,0)$ and
$b_1>0$. By \eqref{theline}, it suffices to prove that
$$
\Big(\int_{\R^{d}} |\widehat{f}(\xi)|^2|\xi|\, \xi_1  \,d\xi\Big)^2>0,
$$
which is the same thing as proving
$$
\Big(\int_{\R^{d}}
|\xi|^{-1}\exp\Big(2\Re(a)|\xi|+2b_1\xi_1+2\Re(c)\Big) \xi_1
\,d\xi\Big)^2>0.
$$
Writing $\R^d_+=\{\,\xi\in\R^d \,:\, \xi_1\ge0\,\}$, the left-hand
side of this inequality is equal to
$$
\Big(\int_{\R^{d}_+}
|\xi|^{-1}\exp\Big(2\Re(a)|\xi|+2\Re(c)\Big)\Big(\exp(2b_1\xi_1)-\exp(-2b_1\xi_1)\Big)\xi_1
\,d\xi\Big)^2,
$$
which is positive, and so we are done.
\end{proof}

We conclude this section by showing how Corollary \ref{thecor} can
be deduced from Corollary \ref{seven} following Foschi
\cite{Foschi}.

\textit{Proof of Corollary~\ref{thecor}.} By Corollary~\ref{seven},
we have the sharp inequality
\begin{equation}\label{zero0}
\|u_+\|_{L^4_{t,x}(\mathbb{R}^{5+1})} \leq (24\pi^{2})^{-1/4}\|\nabla f_+\|_{L^2(\mathbb{R}^5)}
\end{equation}
with equality if $|\xi|\widehat{f_+}(\xi) = \exp(-|\xi|)$, and by
the same argument we also have
\begin{equation}\label{zero}
\|u_-\|_{L^4_{t,x}(\mathbb{R}^{5+1})} \leq (24\pi^{2})^{-1/4}\|\nabla f_-\|_{L^2(\mathbb{R}^5)}
\end{equation}
with equality if $|\xi|\widehat{f_-}(\xi) = \exp(-|\xi|)$.

The space-time Fourier transforms of $u_+^2$, $u_-^2$ and $u_+u_-$
have disjoint supports and therefore
\begin{align}\label{thefir}
\|u\|_{L^4_{t,x}(\mathbb{R}^{5+1})}^4 & = \|u_+^2 + u_-^2 + 2u_+u_-\|_{L^2_{t,x}(\mathbb{R}^{5+1})}^2\nonumber\\
&= \|u_+\|^4_{L^4_{t,x}(\mathbb{R}^{5+1})} +
\|u_-\|_{L^4_{t,x}(\mathbb{R}^{5+1})}^4 +
4\|u_+u_-\|_{L^2_{t,x}(\mathbb{R}^{5+1})}^2.
\end{align}
By the Cauchy--Schwarz inequality,
$$
\|u_+u_-\|_{L^2_{t,x}(\mathbb{R}^{5+1})}^2 \leq \|u_+\|_{L^4_{t,x}(\mathbb{R}^{5+1})}^2\|u_-\|_{L^4_{t,x}(\mathbb{R}^{5+1})}^2
$$
with equality if $|u_+| = |u_-|$. We now apply the basic inequality
$$
2(X^2 + Y^2 + 4XY) \leq 3(X+Y)^2,
$$
which holds for all real numbers $X$ and $Y$, with equality if and
only if $X=Y$, to obtain
$$
\|u\|_{L^4_{t,x}(\mathbb{R}^{5+1})}^4 \le \frac{3}{2}\big(\|u_+\|_{L^4_{t,x}(\mathbb{R}^{5+1})}^2+\|u_-\|_{L^4_{t,x}(\mathbb{R}^{5+1})}^2\big)^2.
$$
Combined with \eqref{zero0} and \eqref{zero}, we see that
$$
\|u\|_{L^4_{t,x}(\mathbb{R}^{5+1})}^4  \leq
\frac{1}{16\pi^{2}}\big(\|\nabla f_+\|_{{L^2(\mathbb{R}^{5})}}^2 +
\|\nabla f_-\|_{{L^2(\mathbb{R}^{5})}}^2\big)^2.
$$
By the parallelogram law,
\begin{equation}\label{thesec}
\|\nabla f_+\|_{L^2(\mathbb{R}^5)}^2 + \|\nabla
f_-\|_{L^2(\mathbb{R}^5)}^2 = \frac{1}{2}\big(\|\nabla
u(0)\|_{L^2(\mathbb{R}^5)}^2 + \|\partial_t u(0)\|_{L^2(\mathbb{R}^5)}^2\big),
\end{equation}
so that
\begin{equation}\label{thiss}
\|u\|_{L^4_{t,x}(\mathbb{R}^{5+1})}^4 \leq \frac{1}{64\pi^{2}}\big(\|\nabla u(0)\|_{L^2(\mathbb{R}^5)}^2 +
\|\partial_t u(0)\|_{L^2(\mathbb{R}^5)}^2\big)^2,
\end{equation}
with equality when $(u(0),\partial_t u(0))$ is such that
$|\xi|\widehat{f_+}(\xi) = |\xi|\widehat{f_-}(\xi) = \exp(-|\xi|)$
(because then $|u_+| = |u_-|$).

It remains to characterise the maximisers. It follows from Corollary
\ref{seven} and the above argument that we have equality in
\eqref{thiss} if and only if
\begin{equation*} \label{e:fplus}
|\xi|\widehat{f_+}(\xi) = \exp(a_+|\xi| + ib_+ \cdot \xi + c_+),
\quad |\xi|\widehat{f_-}(\xi) = \exp(a_-|\xi| + ib_- \cdot \xi +
c_-),
\end{equation*}
and
\begin{equation} \label{e:modulus}
|u_+(t,x)| = |u_-(t,x)| \qquad \text{for almost every $(t,x) \in \mathbb{R} \times \mathbb{R}^5$},
\end{equation}
where $a_+,a_-,c_+,c_-\in \C$, $b_+,b_- \in \mathbb{R}^5$ and
$\text{Re}(a_+),\text{Re}(a_-)<0$. We will see that this is true if
and only if $a_+ = \overline{a_-}$, $b_+ = b_-$ and $\text{Re}(c_+)
= \text{Re}(c_-)$.

To this end we define $\Lambda_{a,b,c}$ by
$$
\Lambda_{a,b,c}(t,x)= \frac{1}{(2\pi)^5} \bigg| \exp(c)
\int_{\mathbb{R}^5} \exp(i(b+x) \cdot \xi + (a+it)|\xi|)
\,\frac{d\xi}{|\xi|} \bigg|.
$$
As $|u_+|=\Lambda_{a_+,b_+,c_+}$ and
$|u_-|=|\overline{u_-}|=\Lambda_{\overline{a_-},b_-,\overline{c_-}}$,
and these functions are  continuous,  we see by \eqref{e:modulus}
that
\begin{equation}\label{theq}
\Lambda_{a_+,b_+,c_+}(t,x) =
\Lambda_{\overline{a_-},b_-,\overline{c_-}}(t,x) \qquad \text{for
each $(t,x) \in \mathbb{R} \times \mathbb{R}^5$}.
\end{equation}
As in \cite{Foschi}, we claim that knowledge of $\Lambda_{a,b,c}$
  uniquely determines $a$, $b$ and $\text{Re}(c)$.  Given \eqref{theq}, it would then
follow that $a_+ = \overline{a_-}, b_+=b_-$ and $\text{Re}(c_+) =
\text{Re}(c_-)$.

Firstly, we note that
$$
\Lambda_{a,b,c}(t,x) \leq \frac{\exp(\text{Re}(c))}{(2\pi)^5}
\int_{\mathbb{R}^5} \exp(\text{Re}(a)|\xi|) \, \frac{d\xi}{|\xi|} =
\Lambda_{a,b,c}(-\text{Im}(a),-b),
$$
so we see that $\Lambda_{a,b,c}$ attains its maximum at
$(-\text{Im}(a),-b)$. Thus, $\text{Im}(a)$ and $b$ are uniquely
determined. Secondly,
$$
\Lambda_{a,b,c}(t-\text{Im}(a),-b) =
\frac{\exp(\text{Re}(c))}{(2\pi)^5} \bigg| \int_{\mathbb{R}^5}
\exp\big((\text{Re}(a) + it)|\xi|\big)\, \frac{d\xi}{|\xi|} \bigg| =
C_0 \frac{\exp(\text{Re}(c))}{|\text{Re}(a) + it|^4}
$$
where $C_0$ is an absolute constant. Thus,
$$
C_0\Lambda_{a,b,c}(t-\text{Im}(a),-b)^{-1} = \exp(-\text{Re}(c))
(\text{Re}(a)^2 + t^2)^2
$$
which is a polynomial in $t$. Since the coefficient of $t^4$ is
$\exp(-\text{Re}(c))$ we have determined $\text{Re}(c)$, and since
the constant term is $\exp(-\text{Re}(c))\,\text{Re}(a)^4$ we have
then determined $\text{Re}(a)$.

It remains to prove that these maximisers can be obtained from
\begin{equation}\label{maxi}
\big(u(0),\partial_tu(0)\big)=\big(0,(1+|\cdot|^2)^{-3}\big)
\end{equation}
under the action of (W1)--(W3) as defined in the introduction. It is
easy to calculate that the ratio
$$
\|u\|_{L^4_{t,x}(\R^{5+1})}\big(\|\nabla
u(0)\|^2_{L^2(\R^5)}+\|\partial_t u(0)\|_{L^2(\R^5)}^2\big)^{-1/2}
$$
is preserved under the action of (W1) and (W2). Appealing to
\eqref{thefir} and \eqref{thesec} we see that the ratio is also
preserved under (W3). We remark that this final invariance does not
hold in general for Strichartz inequalities.

Taking Fourier transforms of the data in \eqref{maxi} we obtain
$$
\big(\widehat{u(0)}(\xi),\widehat{\partial_t u(0)}(\xi)\big) = \big(0,c_0\exp(-|\xi|)\big),
$$
for some $c_0 > 0$, (see for example \cite[pp. 61]{stein}).
Consequently,
\begin{equation} \label{e:flighthome}
\big(|\xi|\widehat{f_+}(\xi),|\xi|\widehat{f_-}(\xi)\big) = \big(\tfrac{1}{2i} c_0 \exp(-|\xi|),
-\tfrac{1}{2i} c_0 \exp(-|\xi|) \big).
\end{equation}
In general, the data $\widehat{f_\pm}(\xi)$ transforms to
$$
\exp( \pm it_0 |\xi| + i x_0 \cdot \xi) \widehat{f_\pm}(\xi), \quad \lambda_1 \lambda_2^{-5} \widehat{f_\pm}(\lambda_2^{-1}\xi), \quad \exp(i \theta_\pm) \widehat{f_\pm}(\xi)
$$
under the action of (W1), (W2), (W3), respectively. It is now
straightforward to check that the pair \eqref{e:flighthome}
transforms under the action of (W1)--(W3) to
$$
\big(|\xi|\widehat{f_+}(\xi),|\xi|\widehat{f_-}(\xi)\big) = \big(\exp(a|\xi| + ib \cdot \xi + c_+),
\exp(\overline{a}|\xi| + ib \cdot \xi + c_- ) \big),
$$
where $\mbox{Re}(c_+) = \mbox{Re}(c_-)$, and so we are done. \qed

\begin{remark} Foschi \cite{Foschi} combined similar arguments with his sharp estimate for the one-sided operator,
$$
\big\| e^{it\sqrt{-\Delta}} f \big\|_{L^{6}_{t,x}(\R^{2+1})} \leq
\Big(\frac{1}{2\pi}\Big)^{1/6} \|f\|_{\dot{H}^{\frac{1}{2}}(\R^2)},
$$
to prove that solutions to the wave equation satisfy
\begin{equation}\label{notsharp}
\|u\|_{L^6_{t,x}(\R^{2+1})} \leq \Big(\frac{25}{64\pi}\Big)^{1/6} \Big(\|u(0)\|^2_{\dot{H}^{\frac{1}{2}}(\mathbb{R}^2)}+\|\partial_t
u(0)\|^2_{\dot{H}^{-\frac{1}{2}}(\mathbb{R}^2)}\Big)^{1/2}.
\end{equation}
He also claimed that the constant in \eqref{notsharp} is attained by
the initial data
$$(u(0),\partial_tu(0))=\big((1+|\cdot|^2)^{-1/2},0\big),$$ modulo
the action of a group of symmetries, however this appears to be
false. In the proof of \eqref{notsharp}, the inequality
$$
|\langle u_+^3,u_+^2u_-\rangle_{t,x}|\le \|u_+^3\|_{L^2_{t,x}(\R^{2+1})}\|u_+^2u_-\|_{L^2_{t,x}(\R^{2+1})},
$$
is used and this holds strictly for such data.
\end{remark}

\section{Proof of Theorem \ref{t:FoschiCarneiro} -- the sharp inequality} \label{section:Carneirowave}

Define the Fourier transform in space and time
$\,\widetilde{\,}\,\,$ by
$$
\widetilde{f}(\tau,\xi)=\int_{\R^{d+1}} f(t,x)\exp\big(-i(t\tau+x\cdot\xi)\big)\,dtdx.
$$
Writing
$$
u_k(t,x) = \prod_{j=1}^k e^{it\sqrt{-\Delta}}f_j(x),
$$
by Plancherel's theorem, we have
$$
\Big\| \prod_{j=1}^k e^{it\sqrt{-\Delta}}f_j\Big\|_{L^2_{t,x}(\mathbb{R}^{d+1})}^2 = (2\pi)^{-(d+1)}\| \widetilde{u_k} \|_
{L^2_{\tau,\xi}(\mathbb{R}^{d+1})}^2.
$$

It is easy to see that,
\begin{eqnarray*}
(e^{it\sqrt{-\Delta}} f_j)^{\thicksim}(\tau,\xi) = 2\pi\delta(\tau - |\xi|)
\widehat{f}_j(\xi),
\end{eqnarray*}
so that defining $\widehat{F} = \otimes_{j=1}^k |\cdot|^{1/2}
\widehat{f}_j$,  we have
$$
\widetilde{u_k}(\tau,\xi) = \frac{1}{(2\pi)^{d(k-1)-1}}\int_{\R^{kd}}
\frac{\widehat{F}(\eta)}{\prod_{j=1}^k |\eta_j|^{1/2}}\, \delta
\Big(\tau - \sum_{j=1}^k |\eta_j|\Big)\delta\Big(\xi - \sum_{j=1}^k \eta_j\Big)
d\eta.
$$
By the Cauchy--Schwarz inequality, this implies that
\begin{equation}\label{klf4}
|\widetilde{u_k}(\tau,\xi)|^2 \leq
\frac{I_k(\tau,\xi)}{(2\pi)^{2d(k-1)-2}} \int_{\R^{kd}} |\widehat{F}
(\eta)|^2 K(\eta)^{2\alpha(k)} \delta\Big(\tau - \sum_{j=1}^k
|\eta_j|\Big)\delta\Big(\xi - \sum_{j=1}^k \eta_j \Big)d\eta,
\end{equation}
where $ \alpha(k) = \frac{(d-1)(k-1)}{2}-1$ and
$$
I_k(\tau,\xi) = \int_{\R^{kd}} \frac{1}{K(\eta)^{2\alpha(k)}
\prod_{j=1}^k |\eta_j|}\,\, \delta \Big(\tau - \sum_{j=1}^k |\eta_j|\Big)
\delta\Big(\xi - \sum_{j=1}^k \eta_j \Big)d\eta.
$$
Crucially, on the support of the delta measures we have
$$
2K(\eta)^2 = \Big( \sum_{j=1}^k |\eta_j| \Big)^2 -
\Big|\sum_{j=1}^k \eta_j \Big|^2  = \tau^2 - |\xi|^2
$$
and therefore
\begin{equation}\label{klf}
I_k(\tau,\xi) = 2^{\alpha(k)}(\tau^2 - |\xi|^2)^{-\alpha(k)} \widetilde{I_k}(\tau,\xi),
\end{equation}
where
$$
\widetilde{I_k}(\tau,\xi) = \int_{\R^{kd}}  \frac{1}{\prod_{j=1}^k
|\eta_j|}\,\,  \delta \Big(\tau - \sum_{j=1}^k |\eta_j|\Big)\delta\Big(\xi -
\sum_{j=1}^k \eta_j \Big) d\eta.
$$

The following lemma was proven by Foschi in the cases $(d,k)=(2,3)$
and $(3,2)$. We  generalise his argument by induction.

\begin{lemma} \label{l:Iwave}
For each $(\tau,\xi)$ with $|\xi| < \tau$, we have
$$
\widetilde{I_2}(\tau,\xi) =
(\tau^2-|\xi|^2)^{\alpha(2)}\frac{|\mathbb{S}^{d-1}|}{2^{d-2}}
$$
and for $k \geq 3$,
$$
\widetilde{I_k}(\tau,\xi) =
(\tau^2-|\xi|^2)^{\alpha(k)}\frac{|\mathbb{S}^{d-1}|^{k-1}}{2^{2\alpha(k)+1}}
\bigg( \prod_{j=2}^{k-1} \emph{B}\big(d-1,\alpha(j) + 1\big) \bigg).
$$
\end{lemma}

\begin{proof}
We begin by recording certain invariances of $\widetilde{I}_k$. Note
that $\widetilde{I_k}$ is the $k$-fold convolution of $\mu$, where
$$
\mu(\tau,\xi) = |\xi|^{-1}\delta(\tau - |\xi|) = 2\delta(\varrho(\tau,\xi))
\chi_{\tau > 0}
$$
and $\varrho : \mathbb{R}^d \rightarrow \mathbb{R}$ is the Minkowski
form given by $\varrho(\tau,\xi) = \tau^2 - |\xi|^2$. It is
well-known, and straightforward to verify, that $\varrho$ is
invariant under Lorentz transformations; that is,
$$
\varrho\big(T_v(\tau,\xi)\big) = \varrho(\tau,\xi)
$$
where the Lorentz transformation $T_v$ is given by
$$
T_v\left[ \begin{array}{cccccccccc} \tau \\ \xi  \end{array}\right] =
\left[ \begin{array}{cccccccccc} \gamma & -\gamma v^t \\ -\gamma v & I_d +
\frac{\gamma-1}{|v|^2} vv^t \end{array} \right] \left[ \begin{array}{ccccccc}
\tau \\ \xi \end{array} \right] = \left[ \begin{array}{ccccccc} \gamma(\tau -
v \cdot \xi) \\ \xi + (\frac{\gamma - 1}{|v|^2}v \cdot \xi - \gamma \tau  )v
\end{array} \right]
$$
for $v \in \mathbb{R}^d$ such that $|v| < 1$, and $\gamma =
(1-|v|^2)^{-1/2}$. Since $|\det T_v| = 1$ it follows that the
$k$-fold convolution of $\mu$ is also invariant under each $T_v$.
Taking $v=-\xi/\tau$, as we may, we have that $\gamma =
(\tau^2-|\xi|^2)^{-1/2}\tau$, and
$$
T_v \left[ \begin{array}{cccccccccc} (\tau^2-|\xi|^2)^{1/2} \\ 0  \end{array}
\right] = \left[ \begin{array}{cccccccccc} \tau \\ -\tau v
\end{array} \right] = \left[ \begin{array}{cccccccccc} \tau \\ \xi
\end{array} \right],
$$
so that
\begin{equation}\label{thefirstbit}
\widetilde{I_k}(\tau,\xi) = \widetilde{I_k}\big((\tau^2-|\xi|^2)^{1/2},0\big).
\end{equation}

Furthermore, by a simple change of variables and homogeneity, for
each $\lambda > 0$,
$$
\widetilde{I_k}(\lambda \tau, \lambda \xi) = \lambda^{2\alpha(k)}
\widetilde {I_k}(\tau,\xi),
$$
where  $\alpha(k) = \frac{(d-1)(k-1)}{2}-1$, so that combined with
\eqref{thefirstbit}, we have
\begin{equation}\label{theabove}
\widetilde{I_k}(\tau, \xi)=(\tau^2-|\xi|^2)^{\alpha(k)}\widetilde{I_k}(1,0).
\end{equation}

Now we are able to compute the desired expression for
$\widetilde{I_2}(\tau,\xi)$ by a direct computation. By
\eqref{theabove} we get
\begin{align*}
\widetilde{I_2}(\tau,\xi) & =  (\tau^2 - |\xi|^2)^{\alpha(2)} \widetilde{I_2}(1,0) \\
& =  (\tau^2 - |\xi|^2)^{\alpha(2)} \int_{\R^{2d}} \delta(1-|\eta_1|-|\eta_2|)
\delta(-
\eta_1-\eta_2) \frac{d\eta_1d\eta_2}{|\eta_1||\eta_2|} \\
& =  (\tau^2 - |\xi|^2)^{\alpha(2)} \int_{\R^{d}} \delta(1-2|\eta_1|) \frac{d\eta_1}
{|\eta_1|^2},
\end{align*}
and hence, by polar coordinates,
\begin{equation} \label{e:I2}
\widetilde{I_2}(\tau,\xi)  =
(\tau^2 - |\xi|^2)^{\alpha(2)} \frac{|\mathbb{S}^{d-1}|}{2^{d-2}},
\end{equation}
as required.

Using \eqref{theabove}, for $k \geq 3$, observe that
\begin{align*}
\widetilde{I_k}(\tau,\xi) & =  (\tau^2 - |\xi|^2)^{\alpha(k)} \widetilde{I_k}(1,0) \\
& = (\tau^2 - |\xi|^2)^{\alpha(k)} \int_{\R^d} \left( \int_{\R^{(k-1)d}}
\delta\Big(1- \sum_{j=1}^k |\eta_j|\Big) \delta\Big(-\sum_{j=1}^k \eta_j\Big)
\frac{d\eta_2\cdots d\eta_k}
{|\eta_2|\cdots|\eta_k|} \right) \frac{d\eta_1}{|\eta_1|}
\end{align*}
and therefore
\begin{equation} \label{e:kfromk-1}
\widetilde{I_k}(\tau,\xi) = (\tau^2 - |\xi|^2)^{\alpha(k)} \int_{|\eta_1| \leq 1/2}
\widetilde{I_{k-1}}(1-|\eta_1|, -\eta_1) \frac{d\eta_1}{|\eta_1|}.
\end{equation}
Using \eqref{e:I2} and \eqref{e:kfromk-1} it follows that
\begin{align*}
\widetilde{I_3}(\tau,\xi) & = (\tau^2 - |\xi|^2)^{\alpha(3)} \frac{|\mathbb{S}^{d-1}|}{2^{d-2}}
\int_{|\eta| \leq 1/2} (1-2|\eta_1|)^{\alpha(2)}\,\frac{d\eta_1}{|\eta_1|} \\
& = (\tau^2 - |\xi|^2)^{\alpha(3)} \frac{|\mathbb{S}^{d-1}|^2}{2^{d-2}}
\int_{0}^{1/2} (1-2r)^{\alpha(2)} r^{d-2} \,dr.
\end{align*}
From this, \eqref{e:kfromk-1} and induction it follows that
\begin{equation*}
\widetilde{I_k}(\tau,\xi) = (\tau^2 -
|\xi|^2)^{\alpha(k)}\frac{|\mathbb{S}^{d-1}|^{k-1}}{2^{d-2}}
\prod_{j=2}^{k-1}  \bigg( \int_0^{1/2} (1-2r)^{\alpha(j)}r^{d-2}\,dr
\bigg)
\end{equation*}
which gives the desired formula for $\widetilde{I_k}(\tau,\xi)$ by a
simple change of variables.
\end{proof}

Combining Lemma~\ref{l:Iwave} with \eqref{klf} we obtain
$$
I_2(\tau,\xi) = 2^{-\frac{d-1}{2}} |\mathbb{S}^{d-1}|
$$
and
$$
I_k(\tau,\xi) = 2^{-(\alpha(k)+1)} |\mathbb{S}^{d-1}|^{k-1}
\prod_{j=2}^{k-1} \text{B}\big(d-1,\alpha(j) + 1\big)
$$
if $k \geq 3$. Substituting into \eqref{klf4}, integrating over
$(\tau,\xi)$, applying Fubini and Plancherel's theorem, we get
$$
\|u_k\|_{L^2_{t,x}(\mathbb{R}^{d+1})}^{2} = (2\pi)^{-(d+1)}\|\widetilde{u_k}\|^2_{L^2_{\tau,\xi}(\mathbb{R}^{d+1})} \leq {\mbox{W} }(d,k)
\int_{\R^{kd}} |\widehat{F}(\eta)|^2 K(\eta)^{2\alpha(k)} \,d\eta,
$$
as required.

We note that if $|\eta_j|\widehat{f}_j(\eta_j) = \exp(a|\eta_j| +
b\cdot \eta_j + c_j)$, where $a,c_1,\ldots,c_k\in \C$, $b\in\C^d$
with $\Re(a)<0$ and $|\Re(b)|<-\Re(a)$, it follows that
$$
\widehat{F}(\eta) = \exp\bigg(a\tau + b \cdot \xi + \sum_{j=1}^k c_j\bigg) \frac{1}{\prod_{j=1}^k |\eta_j|^{1/2}}
$$
on the support of the delta measures. Hence, for such $f_j$ there is
equality in \eqref{klf4} and so the constant ${\mbox{W} }(d,k)$ is
sharp whenever $(d,k) \neq (2,2)$. In the next section we show that
there are no further maximisers, following the approach of Foschi
\cite{Foschi}.

\section{Proof of Theorem \ref{t:FoschiCarneiro} -- characterisation of the
maximisers}\label{characterization}

There is equality in \eqref{klf4} if and only if there exists a
scalar function $\Lambda$  such that
\begin{equation} \label{e:equalitycondition}
 K(\eta)^{\alpha(k)}\widehat{F}(\eta) = \Lambda(\tau,\xi)
K(\eta)^{-\alpha(k)} \prod_{j=1}^k |\eta_j|^{-1/2}
\end{equation}
almost everywhere on the support of the delta measures. Writing
$g_j=|\cdot|\widehat{f}_j$ for all $j=1,\ldots,k$ and $G(\tau,\xi)=(
\tau^2 - |\xi |^2 )^{-\alpha(k)}\Lambda(\tau,\xi)$,
\eqref{e:equalitycondition} implies that the functional equation,
\begin{equation}\label{gw}
\prod_{j=1}^k g_j(\eta_j)= G\Big(\sum_{j=1}^k|\eta_j|, \sum_{j=1}^k \eta_j\Big)
\qquad \text{for almost every $(\eta_1,\ldots,\eta_k)\in (\R^{d})^k$,}
\end{equation}
holds. Since the right-hand side of \eqref{gw} is symmetric in
$\eta_j$ and $\eta_\ell$ it follows that $g_j=\lambda g_\ell$ for
some $\lambda\in\C$. By normalising, we can thus assume that
$g_1=\cdots=g_k=g$.

Note that when $f_1 = \cdots = f_k = f$ and $(d,k) \neq (2,2)$ the
right-hand side of \eqref{e:FoschiCarneiro} is comparable to
\begin{equation} \label{e:RHS}
\| f \|_{\dot{H}^{\frac{1}{2}}(\mathbb{R}^d)}^{2(k-2)} \int_{\mathbb{R}^d} \int_{\mathbb{R}^d} |\widehat{f}(\eta_1)|^2 |\widehat{f}(\eta_2)|^2 |\eta_1|^{\alpha(k)+1} |\eta_2|^{\alpha(k)+1} (1-\eta_1' \cdot \eta_2')^{\alpha(k)} \,d\eta_1 d\eta_2,
\end{equation}
where $\eta_j' = |\eta_j|^{-1} \eta_j$. We claim that when this
quantity is finite then $g$ and $G$ satisfying \eqref{gw} are
continuous, where $g = |\cdot|\widehat{f}$. To see this, first note
that the finiteness of \eqref{e:RHS} implies that $g$ is locally
integrable. Indeed, if $k=2$ and $B$ is any euclidean ball centred
at the origin then, by the Cauchy--Schwarz inequality and the
finiteness of \eqref{e:RHS},
\begin{align*}
\Bigg(\int_{B}|g(\eta)|\,d\eta\Bigg)^2&=\int_{B}\int_{B}|g(\eta_1)||g(\eta_2)|\,d\eta_1d\eta_2\\
&\le C
\Bigg(\int_{B}\int_{B}|\eta_1|^{\frac{5-d}{2}}|\eta_2|^{\frac{5-d}{2}}(1- \eta_1'\cdot\eta_2')^{\frac{3-d}{2}}\,d\eta_1d\eta_2\Bigg)^{1/2} <\infty.
\end{align*}
When $k \geq 3$, we have that $f \in \dot{H}^{1/2}(\mathbb{R}^d)$,
and by a similar argument using the Cauchy--Schwarz inequality, we
get that $g$ is locally integrable.

Now, since \eqref{gw} holds for $g$ and $G$ we have that
$$
g(\eta_1)g(\eta_2) = \widetilde{G}(|\eta_1| + |\eta_2|, \eta_1 + \eta_2)
$$
for almost every $(\eta_1,\eta_2)\in \R^{d} \times \mathbb{R}^d$,
where $\widetilde{G}$ is equal to $G$ modulo composition with
certain translations and multiplication by nonzero scalars
(depending on $g$). The continuity of $g$ and $G$ now follows from
Lemma~7.20 (which in fact holds for all $d\ge 2$) and
Proposition~7.5 of \cite{Foschi}.

Hence, it suffices to characterise all solutions to the functional
equation
\begin{equation}\label{e:functionalequation}
g(\eta_1)g(\eta_2) = G(|\eta_1| + |\eta_2|, \eta_1 + \eta_2)
\qquad \text{for each $(\eta_1,\eta_2) \in \R^{d} \times \mathbb{R}^d$,}
\end{equation}
where $g$ and $G$ are continuous. In this case, we may assume $g(0)
\neq 0$. Otherwise \eqref{e:functionalequation} gives
\begin{equation*}\label{thesecond}
G(|\eta|,\eta) = g(\eta)g(0)=0
\end{equation*}
for all $\eta \in \R^d$, which, combined with
\eqref{e:functionalequation} again, implies that
$$
g(\eta)^2 = G(2|\eta|,2\eta) = G(|2\eta|,2\eta)=0
$$
for all $\eta\in\R^d$, and this is the trivial case.

Noting that $G(0,0)=g(0)^{2}\neq 0$, we can rewrite
\eqref{e:functionalequation} as
$$
H(|\eta_1|,\eta_1)H(|\eta_2|,\eta_2) = H(|\eta_1| + |\eta_2|,\eta_1+\eta_2) \qquad \text{for each $(\eta_1,\eta_2)\in \R^{d} \times \mathbb{R}^d$},
$$
where $H(\tau,\xi)=G(0,0)^{-1}G(\tau,\xi)$. By algebraic properties
of the cone (see \cite[Lemma 7.18]{Foschi}), this implies
$$
H(X)H(Y)=H(X+Y) \qquad \text{for all $X,Y\in\{\,(\tau,\xi)\in\R^{d+1}\,:\, \tau>|\xi|\,\}$.}
$$
Thus, by \cite[Lemma~7.1]{Foschi}, there exists $a\in\C$ and
$b\in\C^{d}$ such that
$$
H(\tau,\xi)=\exp(a\tau+b\cdot\xi)
$$
for $(\tau,\xi)$ in the solid cone. Choosing $c\in\C$ such that
$\exp(2c)=G(0,0)$, we obtain that
$G(\tau,\xi)=\exp(a\tau+b\cdot\xi+2c)$, so that $g(
\xi)=\exp(a|\xi|+b\cdot\xi+c)$. Thus, the Fourier transforms of the
maximisers $f$ take the form
$$
\widehat{f}(\xi)=|\xi|^{-1}\exp(a|\xi|+b\cdot\xi+c).
$$

It remains to check, under which conditions on $a$ and $b$, the
right-hand side of \eqref{e:FoschiCarneiro}, or equivalently the
quantity \eqref{e:RHS}, is finite. It is easy to see that $\Re(a)<
0$ is necessary and $c \in \mathbb{C}^d$ has no effect on such
considerations. So we assume $\Re(a)< 0$ and $c = 0$ from now on,
and consider the cases $\alpha(k)=0$ and $\alpha(k)>0$ separately.

When $\alpha(k)=0$ the quantity \eqref{e:RHS} is equal to the $k$th
power of
$$
\|f\|_{\dot{H}^{\frac{1}{2}}(\mathbb{R}^d)}^2 = \int_{\R^{d}} \exp(2\Re (a)|\xi|+2\Re(b)\cdot\xi) \,\frac{d\xi}{|\xi|}.
$$
Using polar coordinates this is equal to a constant multiple of
$$
\int_{-1}^1 \int_0^\infty \exp \Big( (2\Re (a) + 2|\Re (b)| u)r \Big) r^{d-2} (1-u^2)^{\frac{d-3}{2}} \,drdu,
$$
which is finite if and only if $|\Re (b)| < -\Re (a)$.

When $\alpha(k)>0$, the quantity \eqref{e:RHS} is bounded above by
$$
\|f\|_{\dot{H}^{\frac{1}{2}}(\mathbb{R}^d)}^{2(k-2)} \|f\|_{\dot{H}^{\frac{\alpha(k)+1}{2}}(\mathbb{R}^d)}^4
$$
and using polar coordinates as above it is easy to check that if
$|\Re(b)|<-\Re(a)$ then $\|f\|_{\dot{H}^{s}(\mathbb{R}^d)} < \infty$
for any $s \geq 1/2$. Also, we have shown above that if $|\Re(b)|
\ge -\Re(a)$ then $\|f\|_{\dot{H}^{1/2}(\mathbb{R}^d)}$ is not
finite, and therefore the quantity \eqref{e:RHS} is not finite if $k
\geq 3$. Hence it remains to show that for $k=2$, $d \geq 4$ and
$|\Re(b)| \ge -\Re(a)$ the quantity \eqref{e:RHS} is not finite. In
this case, \eqref{e:RHS} is equal to
\begin{align} \label{e:RHS'}
\int_{\mathbb{R}^d} \int_{\mathbb{R}^d} \exp(2\Re (a)|\eta_1| + 2\Re (b) \cdot \eta_1) & \exp(2\Re (a)|\eta_2| + 2\Re (b) \cdot \eta_2) \notag  \\
& |\eta_1|^{\frac{d-5}{2}}|\eta_2|^{\frac{d-5}{2}}(1-\eta_1' \cdot \eta_2')^{\frac{d-3}{2}} \,d\eta_1 d\eta_2.
\end{align}
Let $\Omega \subset \mathbb{S}^{d-1}$ be a closed cap centred at
$\Re (b)'$ so that $\Re (b)' \cdot \eta_2' \geq 1-\varepsilon$ for
each $\eta_2' \in \Omega$.  Then, for each $\eta_1$ with $\eta_1'
\in \Omega^\perp$, we have that the $d\eta_2$-integral in
\eqref{e:RHS'} is bounded below by a constant multiple of
\begin{equation*}\label{bitty}
\int_{\Omega} \int_0^\infty  \exp
\Big( (2\Re (a) + 2\Re (b) \cdot \eta_2')r \Big)
r^{\frac{3d-7}{2}} \,drd\sigma(\eta_2'),
\end{equation*}
which is not finite when $|\Re(b)| \ge -\Re(a)$, and hence neither
is \eqref{e:RHS}. This completes the proof Theorem
\ref{t:FoschiCarneiro}.

\section{The Schr\"odinger equation : Carneiro's inequality revisited} \label{section:Carneiro}

The following theorem is the natural analogue of Theorem \ref{fq2}
for the Schr\"odinger evolution operator and is due to Carneiro
\cite{Carneiro}.

\begin{theorem} \label{t:Carneiro}
Let $d\ge 2$. Then the inequality
\begin{equation*}
\big\| e^{it\Delta} f_1\,e^{it\Delta} f_2 \big\|^{2}_{L^{2}_{t,x}(\mathbb{R}^{d+1})}
\leq \emph{S}(d,2) \int_{\R^ {2d}}
|\widehat{f}_1(\xi_1)|^2|\widehat{f}_2(\xi_2)|^2|\xi_1 -
\xi_2|^{d-2} \,d\xi_1d\xi_2
\end{equation*}
holds with sharp constant given by
$$
\emph{S}(d,2) = 2^{-d} (2\pi)^{-3d+1} |\mathbb{S}^{d-1}|
$$
which is attained if and only if
$$\widehat{f}_j(\xi)=\exp(a|\xi|^2+b\cdot\xi+c_j),$$ where $a,c_1,c_2 \in \C$,
$b\in\C^d$ and $\Re (a)<0$.
\end{theorem}
The case $d=1$ is special because, for $f_1$ and $f_2$ with
separated Fourier support, we have the identity\footnote{the authors
thank Luis Vega for bringing this to their attention}
\begin{equation} \label{e:Schroidentity}
\big\| e^{it\Delta} f_1\,e^{it\Delta} f_2 \big\|^{2}_{L^{2}_{t,x}(\mathbb{R}^{1+1})}
= \frac{1}{2(2\pi)^2} \int_{\R^{2}}
|\widehat{f}_1(\xi_1)|^2|\widehat{f}_2(\xi_2)|^2 \,\frac{d\xi_1d\xi_2}{|\xi_1 -
\xi_2|},
\end{equation}
which follows easily by changes of variables and Plancherel's
theorem. The estimate is implicit in the thesis of
Fefferman~\cite{fest} and the identity is evident from the
calculation in \cite[pp. 412]{bigstein} (see also \cite[Section
17]{fokl} for an analagous inequality for the extension operator on
$\mathbb{S}^1$). The interaction weight is  too singular for the
right-hand side of \eqref{e:Schroidentity} to be finite for
integrable $f_1=\lambda f_2\neq 0$. A manifestation of this is that
$\text{S}(1,2) = (2\pi)^{-2}$ is equal to twice the constant arising
in the identity~\eqref{e:Schroidentity}.

With $d=2$, the power of the interaction weight is zero, and so the
estimate reduces to the sharp version of the Strichartz estimate
\eqref{stricschro} due to Foschi \cite{Foschi}.

For the case $d=4$, Carneiro deduced the following corollary from
Theorem \ref{t:Carneiro} in the same way that the inequality of
Corollary~\ref{seven} was deduced from
Theorem~\ref{t:FoschiCarneiro}.

\begin{corollary}\label{c:Carneiro}
Let $d=4$. Then
\begin{equation}\label{canny}
\big\| e^{it\Delta} f \big\|_{L^{4}(\R^{4+1})} \leq
\big(32\pi\big)^{-1/4} \|f\|_{L^2(\R^4)}^{1/2}\|\nabla
f\|_{L^2(\R^4)}^{1/2}.
\end{equation}
The constant is sharp and is attained if and only if
$$\widehat{f}(\xi) = \exp({a|\xi|^2+ib\cdot\xi+c}),$$ where
$a,c\in\C$, $b\in\R^d$ and $\Re (a) <0$.
\end{corollary}

Note that the class of maximisers is smaller than that of Theorem
\ref{t:Carneiro} (although in~\cite{Carneiro} it was suggested
otherwise\footnote{E. Carneiro thanks the authors for pointing out
this minor oversight in Corollary 3 of [7]}). The maximisers can be
obtained from $u(0) = \exp(-|\cdot|^2)$ under the action of the
group generated by:
\begin{itemize}
\item[(S1)] $u(t,x)\to u(t+t_0,x+x_0)$ with $t_0\in\R$, $x_0\in\R^d$,
\item[(S2)] $u(t,x)\to \lambda_1 u(\lambda_2^2 t,\lambda_2 x)$ with
$\lambda_1,\lambda_2
>0$.
\item[(S3)] $u(t,x)\to e^{i\theta} u(t,x)$ with $\theta\in\R$.
\end{itemize}

Another well-known symmetry for the Schr\"odinger equation is the
Galilean transformation:
\begin{itemize}
\item[(S4)] $u(t,x)\to \exp\big(-i(x\cdot v+|v|^2t)\big)u(t,x+2v)$ with
$v\in\R^d$.
\end{itemize}

Foschi \cite{Foschi} proved that the maximisers for
\eqref{stricschro} with $d=1,2$ are given by
$$\widehat{f}(\xi) = \exp({a|\xi|^2+b\cdot\xi+c}),$$ where
$a,c\in\C$, $b\in\C^d$ and $\Re (a) <0$, which is a larger class
than that of Corollary~\ref{c:Carneiro}. This is explained by the
fact that
$$
\| e^{it\Delta} f \|_{L_{t,x}^{p}(\mathbb{R}^{d+1})} \|f\|_{L^2(\mathbb{R}^d)}^{-1}
$$
is invariant under the action of (S4), whereas
$$
\| e^{it\Delta} f \|_{L_{t,x}^{p}(\mathbb{R}^{d+1})} \|f\|_{L^2(\mathbb{R}^d)}^{-1/2}\|\nabla f\|_{L^2(\mathbb{R}^d)}^{-1/2}
$$
is not.

For interest and completeness we provide an alternative proof of the
estimate in Theorem \ref{t:Carneiro}, following Foschi
\cite{Foschi}.

\textit{Proof of Theorem \ref{t:Carneiro}}. By the change of
variables $t\to -t$, we can replace the operators $e^{it\Delta}$ by
$e^{-it\Delta}$. Writing
$w(t,x)=e^{-it\Delta}f_1(x)e^{-it\Delta}f_2(x)$, we have
$$
\widetilde{w}(\tau,\xi) = \frac{1}{(2\pi)^{d-1}}\int_{\R^{2d}} \widehat{f_1}(\eta_1)
\widehat{f_2}(\eta_2) \, \delta(\tau -|\eta_1|^2 - |\eta_2|^2)
\delta(\xi - \eta_1 - \eta_2) \,d\eta
$$
because
\begin{eqnarray*}
(e^{-it\Delta} f_j)^{\thicksim}(\tau,\xi) = 2\pi\delta(\tau - |\xi|^2)
\widehat{f}(\xi).
\end{eqnarray*}
Therefore, by the Cauchy--Schwarz inequality,
$|\widetilde{w}(\tau,\xi)|^2$ is majorised by
\begin{equation}\label{klf2}
\frac{I(\tau,\xi)}{(2\pi)^{2(d-1)}} \int_{\R^{2d}} |\widehat{f_1}
(\eta_1)|^2 |\widehat{f_2}(\eta_2)|^2 |\eta_1-\eta_1|^{d-2} \,
\delta(\tau - |\eta_1|^2 - |\eta_2|^2) \delta(\xi - \eta_1 - \eta_2)
\,d\eta,\end{equation} where
$$
I(\tau,\xi) = \int_{\R^{2d}} |\eta_1-\eta_2|^{-(d-2)}\,
\delta(\tau - |\eta_1|^2 - |\eta_2|^2)
\delta(\xi - \eta_1 - \eta_2) \,d\eta.
$$
As before, we shall prove that $I(\tau,\xi)$ is constant.

Firstly, note that on the support of the delta measures, we have
$$
|\eta_1 - \eta_2|^2 = 2(|\eta_1|^2 + |\eta_2|^2) - |\eta_1 + \eta_2|^2
= 2\tau - |\xi|^2.
$$
Therefore
\begin{equation}\label{this} I(\tau,\xi) = (2\tau - |\xi|^2)^{-\frac{d-2}{2}}
\widetilde{I}(\tau,\xi),
\end{equation}
where
$$
\widetilde{I}(\tau,\xi) = \int_{\R^{2d}} \delta(\tau - |\eta_1|^2 - |\eta_2|^2)
\delta(\xi - \eta_1 - \eta_2)\,d\eta = \mu * \mu (\tau,\xi)
$$
and the measure $\mu$ is given by
$$
\mu(\tau,\xi) = \delta(\tau - |\xi|^2).
$$
For each $v \in \mathbb{R}^d$, it is easy to see that $\mu$ is
invariant under the affine map
$$
\left[ \begin{array}{cccccccccc} \tau \\ \xi  \end{array}
\right] \mapsto A_v\left[ \begin{array}{cccccccccc} \tau \\ \xi  \end{array}
\right] + b = \left[ \begin{array}{cccccccccc} \tau + 2\xi \cdot v + |v|^2 \\ \xi +v \end{array}
\right],
$$
where
$$
A_v = \left[ \begin{array}{ccccccccc}  1  & 2v^t \\ 0 & I_d
\end{array} \right] \qquad \text{and} \qquad b = \left[ \begin{array}{cccccccccc}
|v|^2 \\ v  \end{array}
\right].
$$
Since $|\det A_v| = 1$ it follows that
$$
\widetilde{I}(\tau,\xi) =
\widetilde{I}(A_v(\tau,\xi) + 2b) = \widetilde{I} (\tau + 2\xi \cdot
v + 2|v|^2, \xi + 2v).
$$
Choosing $v = -\xi/2$ and by scaling we get
\begin{equation}\label{this2}
\widetilde{I}(\tau,\xi) = \widetilde{I}(\tau - \tfrac{1}{2}|\xi|^2,0) = 2^{-\frac{d-2}{2}}(2\tau-|\xi|^2)^{\frac{d-2}{2}}\widetilde{I}(1,0).
\end{equation}
Finally, a straightforward calculation gives
$$
\widetilde{I}(1,0) = 2^{-\frac{d+2}{2}}|\mathbb{S}^{d-1}|.
$$
Combining this with \eqref{this} and \eqref{this2}, it follows that
$$
I(\tau,\xi) = 2^{-d}|\mathbb{S}^{d-1}|.
$$

Now by Plancherel's theorem, \eqref{klf2} and Fubini, we get
\begin{equation*}\label{last}
\|w\|^{2}_{L^2_{t,x}(\mathbb{R}^{d+1})} = \frac{1}{(2\pi)^{d+1}}\|\widetilde{w}\|^2_{L^2_{\tau,\xi}(\mathbb{R}^{d+1})} \leq
\mbox{S}(d,2)\int_{\R^{2d}} |\widehat{f_1}(\eta_1)|^2 |\widehat{f_2}(\eta_2)|^2
|\eta_1 - \eta_2|^{d-2} \,d\eta,
\end{equation*}
as required.

If $\widehat{f}_j(\eta_j) = \exp(a|\eta_j|^2 + b\cdot \eta_j +
c_j)$, where $a,c_1,c_2 \in \mathbb{C}, b \in \mathbb{C}^d$ with
$\mbox{Re}(a) < 0$, then
$$
\widehat{f}_1(\eta_1)\widehat{f}_2(\eta_2) = \exp (a\tau + b\cdot \xi + c_1 + c_2)
$$
on the support of the delta measures. Hence, for such $f_j$ there is
equality in the application of Cauchy--Schwarz in \eqref{klf2} which
implies that the constant $\mbox{S}(d,2)$ is sharp.

As before, if $f_1,f_2$ are maximisers then it is easy to see that
$f_1=\lambda f_2$ for some $\lambda \in \C$. The characterisation of
maximisers in Theorem \ref{t:Carneiro} now follows from
\cite{Carneiro} where this was done for the linear estimate (i.e.
Theorem \ref{t:Carneiro} with $f_1 = f_2$). \qed


To conclude we present a $k$-linear generalisation of Theorem
\ref{t:Carneiro} which is analogous to Theorem
\ref{t:FoschiCarneiro}.

\begin{theorem} \label{t:last}
Suppose that $d,k \geq 2$ and let $ \beta(k) = \frac{d(k-1)}{2}-1$.
Let $K : (\mathbb{R}^d)^k \rightarrow [0,\infty)$ be given by
$$
K(\eta) = \bigg( \sum_{1 \leq i < j \leq k}
|\eta_i - \eta_j|^2 \bigg)^{1/2}
$$
for $\eta = (\eta_1,\ldots,\eta_k) \in (\mathbb{R}^d)^k$. Then the
inequality
\begin{equation} \label{e:FoschiCarneiroSchro}
\Big\| \prod_{j=1}^k e^{it\Delta}f_j
\Big\|_{L^{2}_{t,x}(\mathbb{R}^{d+1})}^{2} \leq \emph{S}(d,k) \int_{\R^{kd}}
\prod_{j=1}^k|\widehat{f}_j(\eta_j)|^2 \,K(\eta)^{2\beta(k)} \,d\eta
\end{equation}
holds with sharp constant given by
$$
\emph{S}(d,k) = \pi(2\pi)^{-d(2k-1)} k^{-\frac{dk}{2}+1} |\mathbb{S}^{(k-1)d-1}|
$$
which is attained if and only if
\begin{equation*}|\xi|\widehat{f_j}(\xi)=\exp(a|\xi|+b\cdot\xi+c_j),\end{equation*}
where $a,c_1,\ldots,c_k\in \C$, $b\in\C^d$, $\Re(a)<0$ and
$|\Re(b)|<-\Re(a)$.
\end{theorem}

Theorem \ref{t:last} was proven by Carneiro in \cite{Carneiro}
following the argument of Hundertmark--Zharnitsky \cite{huza}. We
remark that it is possible to generalise the above proof of Theorem
\ref{t:Carneiro} to prove \eqref{e:FoschiCarneiroSchro}.

\vspace{1em}

The authors wish to thank Jon Bennett and Thomas Duyckaerts for
helpful conversations. They also thank Sanghyuk Lee and Seoul
National University, where part of this research was conducted, for
their warm hospitality.


\begin{thebibliography}{MMM}


\bibitem{bage} H. Bahouri and P. G\'erard, {\it High frequency approximation of
solutions to critical nonlinear wave equations}, Amer. J. Math.
\textbf{121} (1999), 131--175.


\bibitem{ba} B. Barcel\'o, {\it On the restriction of the Fourier transform to a conical surface},
Trans. Amer. Math. Soc. {\bf 292} (1985), 321--333.

\bibitem{beva} P. B\'egout\ and\ A. Vargas, \textit{Mass concentration phenomena for the
$L\sp 2$-critical nonlinear Schr\"odinger equation}, Trans. Amer.
Math. Soc. {\bf 359} (2007), 5257--5282.

\bibitem{bebecahu} J. Bennett, N. Bez, A. Carbery and D.
Hundertmark, \textit{Heat-flow monotonicity of Strichartz norms},
 Anal. PDE  \textbf{2}  (2009),  147--158.

\bibitem{bo} J. Bourgain, \textit{Estimates for cone multipliers},
Geometric Aspects of Functional Analysis, Oper. Theory Adv. Appl.
\textbf{77}, Birkh\"auser (1995), 41--60.




\bibitem{bulut} A. Bulut, \textit{Maximizers for the Strichartz inequalities for the
wave equation}, arXiv:0905.1678.


\bibitem{Carneiro} E. Carneiro, \textit{A sharp inequality for the Strichartz norm},
Int. Math. Res. Not. {\bf 16} (2009), 3127--3145.

\bibitem{chsh} M. Christ and S. Shao, \textit{Existence of extremals for a Fourier restriction
inequality}, arXiv:1006.4319.

\bibitem{chsh2} \bysame, \textit{On the extremizers of an adjoint Fourier restriction inequality}, arXiv:1006.4318.

\bibitem{cake} R. Carles and S. Keraani, {\it On the role of quadratic oscillations
in nonlinear Schr¨odinger equations. II. The $L^2$-critical case},
Trans. Amer. Math. Soc. \textbf{359} (2007), 33--62.


\bibitem{dumero} T. Duyckaerts,  F. Merle\ and\ S. Roudenko, \textit{Maximizers for the Strichartz norm for small
solutions of mass-critical NLS},  Ann. Sc. Norm. Super. Pisa Cl.
Sci., to appear.

\bibitem{favevi} L. Fanelli, L. Vega\ and N. Visciglia, \textit{On the existence of maximizers for a family of restriction theorems}, arXiv:1007.2063.

\bibitem{fest} C. Fefferman, \textit{Inequalities for strongly singular convolution operators}, Acta Math. {\bf 124} (1970), 9--36.

\bibitem{Foschi} D. Foschi, \textit{Maximizers for the Strichartz inequality}, J. Eur. Math. Soc. {\bf 9}
(2007), 739--774.


\bibitem{fokl} D. Foschi and S. Klainerman, \textit{Bilinear space-time estimates for homogeneous wave equations},
Ann. Sci. \'{E}cole Norm. Sup. (4)  \textbf{33}  (2000), 211--274.

\bibitem{huza} D. Hundertmark and V. Zharnitsky, \textit{On sharp
Strichartz inequalities in low dimensions}, Int. Math. Res. Not.
(2006), Art. ID 34080, 18 pp.


\bibitem{keme} C.E. Kenig and F. Merle, \textit{Global well-posedness,
scattering and blow-up for the energy critical focusing non-linear
wave equation},  Acta Math.  \textbf{201}  (2008) 147--212.

\bibitem{klma} S. Klainerman and M. Machedon, \textit{Space-time estimates for null
forms and the local existence theorem}, Comm. Pure Appl. Math.
\textbf{46} (1993), 1221--1268.

\bibitem{km2} \bysame,
\emph{Remark on Strichartz--type inequalities.} With appendices by
Jean Bourgain and Daniel Tataru. Internat. Math.  Res. Notices
\textbf{5} (1996), 201--220.

\bibitem{km3}
\bysame, \emph{On the regularity properties of a model problem
related to wave maps}, Duke Math. J. \textbf{87} (1997), 553--589.

\bibitem{ku} M. Kunze, \textit{On the existence of a maximizer for the Strichartz
inequality}, Comm. Math. Phys. \textbf{243} (2003), 137--162.

\bibitem{lerova} S. Lee, K.M. Rogers, and A. Vargas, \textit{Sharp null form
estimates for the wave equation in $\Bbb R^{3+1}$},  Int. Math. Res.
Not. IMRN 2008, Art. ID rnn 096, 18 pp.

\bibitem{leva} S. Lee and A. Vargas, \textit{Sharp null form estimates for the wave
equation},  Amer. J. Math.  \textbf{130}  (2008), 1279--1326.

\bibitem{meve} F. Merle and L. Vega, \textit{Compactness at blow-up time for $L^2$ solutions of the critical nonlinear
Schr¨odinger equation in 2D}, Internat. Math. Res. Notices (1998),
no. 8, 399--425.



\bibitem{shao} S. Shao, \textit{Maximizers for the Strichartz and the Sobolev--Strichartz
inequalities for the Schr\"odinger equation}, Electron. J.
Differential Equations (2009), No. 3, 13 pp.

\bibitem{stein} E.M. Stein, {\it Singular integrals and differentiability properties of
functions}, Princeton University Press (1970).

\bibitem{bigstein} E.M. Stein, \textit{Harmonic Analysis}, Princeton
Univ. Press, Princeton, N. J. (1993).


\bibitem{st} R.S. Strichartz, \textit{Restrictions of Fourier transforms to quadratic
surfaces and decay of solutions of wave equations}, Duke Math. J.
{\bf 44} (1977), 705--714.

\bibitem{ta0} T. Tao, \textit{Endpoint bilinear restriction theorems for the cone,
and some sharp null form estimates},  Math. Z.  \textbf{238} (2001),
215--268.








\bibitem{wo} T. Wolff, \textit{A sharp bilinear cone restriction estimate}, Ann. of Math.
(2) \textbf{153} (2001), 661--698.

\end{thebibliography}
\end{document}